\documentclass[12pt]{amsart}

\usepackage{amsfonts}
\usepackage{amssymb}
\usepackage[letterpaper, left=2.5cm, right=2.5cm, top=2.5cm,
bottom=2.5cm,dvips]{geometry}
\usepackage{verbatim}
\usepackage[T1]{fontenc}
\usepackage{graphicx}
\usepackage{amsmath}

\newtheorem{theorem}{Theorem}
\theoremstyle{plain}

\newtheorem{case}{Case}

\newtheorem{conjecture}{Conjecture}
\newtheorem{corollary}{Corollary}

\newtheorem{proposition}{Proposition}

\newtheorem{dupa}{Case}
\numberwithin{equation}{section}
\usepackage{diagbox} 

\usepackage{array,multirow}

\usepackage{xcolor}

\usepackage[backref]{hyperref}
\hypersetup{
	colorlinks,
	linkcolor={blue!60!black},
	citecolor={red!60!black},
	urlcolor={red!60!black}
}

\begin{document}
	\title{Variations on shuffle squares}
	
	\author{Jaros\l aw Grytczuk}
	\address{Faculty of Mathematics and Information Science, Warsaw University
		of Technology, 00-662 Warsaw, Poland}
	\email{j.grytczuk@mini.pw.edu.pl}
	
	\author{Bart\l omiej Pawlik}
	\address{Institute of Mathematics, Silesian University of Technology, 44-100 Gliwice, Poland}
	\email{bpawlik@polsl.pl}
	
	\author{Mariusz Pleszczy\'nski}
	\address{Institute of Mathematics, Silesian University of Technology, 44-100 Gliwice, Poland}
	\email{mariusz.pleszczynski@polsl.pl}

	\begin{abstract}
		We study decompositions of words into subwords that are in some sense similar, which means that one subword may be obtained from the other by a relatively simple transformation. Our main inspiration are \emph{shuffle squares}, an intriguing class of words arising in various contexts, from purely combinatorial to more applied, like modeling concurrent processes or DNA sequencing. These words can be split into two parts that are just identical. For example, ${\color{red}{ACT}\color{blue}{AC}\color{red}{A}\color{blue}{TAG}\color{red}{G}}$ is a shuffle square consisting of two copies of the word $ACTAG$. Of course, each letter must appear any even number of times in each shuffle square. We call words with that property \emph{even}.

We mainly discuss new problems concerning generalized shuffle squares. We propose a number of conjectures and provide some initial results towards them. For instance, we prove that every binary word is a \emph{cyclic shuffle square}, meaning that it splits into two subwords, one of which is a~cyclic permutation of the other. The same statement is no longer true over larger alphabets, but it seems plausible that a similar property should hold with slightly less restricted permutation classes (depending on the alphabet size). For instance, we conjecture that every even \emph{ternary} word is a \emph{dihedral} shuffle square, which means that it splits into two subwords, one of which can be obtained from the other by a permutation corresponding to the~symmetry of a~regular polygon. We propose a general conjecture stating that a linear number of permutations is sufficient to express all even $k$-ary words as generalized shuffle squares.

Our discussion is complemented by some enumerative and computational experiments. In particular, we disprove our former conjecture stating that every even binary word can be turned into a shuffle square by a cyclic permutation. The smallest counterexample has length $24$. We call words of this type \emph{shuffle anti-squares}. We determined all of them up to the length $28$.
	\end{abstract}

	\maketitle
	
	\section{Introduction}
	
	The phenomena of repeatability have fascinated researchers in various fields of science since time immemorial. One spectacular example in mathematics, from the beginning of the last century, is the work of Axel Thue \cite{Thue,Thue2} on repetitions in words. His discoveries provided the basis for the theory of formal languages, which in turn constituted the foundation of computer science (see \cite{BerstelThue,Lothaire}). One of these results can be stated in simple terms as follows.
	
	Notice first that every binary word of length at least four must contain one of the four blocks, $\mathtt{00}$, $\mathtt{11}$, $\mathtt{0101}$, or $\mathtt{1010}$, as a contiguous segment. Indeed, to avoid a single-bit repetition in a sequence of bits, we are forced to alternate them, which leads to the repetition of $\mathtt{01}$ or $\mathtt{10}$ in the fourth step. Surprisingly, if we have just one more symbol in the alphabet, then, as proved by Thue \cite{Thue}, we are able to construct arbitrarily long \emph{ternary} words without segments of such repetitive form of any possible length. For instance, the word $$\mathtt{012102120102012021201021012}$$ has no segment of the form $SS$, whatever $S$ may be. Due to their algebraic connotations, segments of this repetitive type are called \emph{squares}, while words avoiding them are called \emph{square-free}.
	
	Thue's theorem inspired a huge amount of research, giving birth to an entire mathematical discipline known as \emph{combinatorics on words} (see \cite{BeanEM,BerstelPerrin,Lothaire,LothaireAlgebraic}. It also has many important connections and applications, not only in pure mathematics or theoretical computer science. For instance, the same structure of repeated segments in sequences is an object of intensive studies in genetics. Indeed, one of fundamental mechanisms of gene mutation, known as \emph{tandem duplication}, relies on copying a segment of a DNA sequence and inserting the resulting copy right after the original template, creating in this way a \emph{tandem repeat} (a square) (see \cite{AlonBFJ,JainHSB,Zhang}). For instance, by copying the underlined segment, the sequence $GTG\underline{CAT}CTGA$ mutates into the form $GTG\underline{CATCAT}CTGA$ with the underlined tandem repeat. It is believed that in general, gene duplication is the main driving force behind the evolution of organisms, where the majority of essential mutations are believed to be of the tandem duplication type. In particular, around $3\%$ of the human genome are formed of tandem repeats (see \cite{Lander}).

Our main inspiration in the present paper is a more general variant of repetitive structure in words, known as \emph{shuffle squares}. These are words that can be split into two copies of the same word. For instance, $\color{red}{AT}\color{blue}{A}\color{red}{C}\color{blue}{T}\color{red}{G}\color{blue}{CG}$ is a shuffle square consisting of two copies of the word $ATCG$. Notice that any shuffle square must be an \emph{even word}, that is, a word in which every letter occurs an even number of times. Let us mention that an analog of the result of Thue for shuffle squares is known to hold over a $6$-letter alphabet, as proved recently by Bulteau, Jug\'{e}, and Vialette \cite{Bul2023}. This means that there are arbitrarily long words built of six different letters not containing shuffled squares as segments. The least number of letters needed for that property is however not known.

The study of word shuffling in formal language theory was initiated by Ginsburg and Spanier [6]. However, further investigations were also motivated by applications to modeling sequential execution of concurrent processes (see \cite{Riddle1,Riddle2,Shaw}). Shuffle squares were introduced by Henshall, Rampersad and Shallit \cite{Hen2012}, who mainly studied the problem of their enumeration. This problem seems hard. Even for the smallest binary case, the counting function is not known, even asymptotically. It was only recently proved by He, Huang, Nam, and Thaper \cite{He2024} that the number of binary shuffle squares of length $2n$ is at least ${2n\choose n}$ for $n\geqslant3$. An intriguing conjecture posed in \cite{He2024} states that the vast majority of even binary words are shuffle squares, in a sense that their fraction among even words of fixed length tends to unity with the length growing to infinity.

Another striking result about shuffle squares was obtained by Axenovich, Person, and Puzynina \cite{Axenovich-Person-Puzynina}. It states that every binary word of length $n$ is a shuffle square up to the deletion of $o(n)$ letters. A precise asymptotic formula for the maximum number of letters to be deleted is not known. It is also not known if a similar statement is true for ternary words.

It also comes naturally to study computational problems for shuffle squares. It was proved by Buss and Soltys \cite{Buss-Soltys} that recognizing shuffle squares is NP-complete for sufficiently large fixed alphabet. Recently, Bulteau and Vialette \cite{Bul2020} demonstrated that the same holds for a binary alphabet.

In this paper we introduce a generalization of shuffle squares, in which we are satisfied with splitting of a given word into two subwords that are sufficiently \emph{similar}. The degree of similarity of words is measured by complexity of a permutation needed to transform one word into another. For instance, two words are \emph{cyclically similar} if one is a cyclic shift of the other.

Our main result (Theorem \ref{Theorem Cyclic Shuffle Squares}) states that every even binary word is a \emph{cyclic shuffle square}, which means that it can be split into two cyclically similar subwords. This statement may seem surprising, but the proof is very simple by using the famous Necklace Splitting Theorem, due to Goldberg and West \cite{Goldberg-West} (see \cite{Alon-West} for a short and elegant proof based on the Borsuk-Ulam Theorem).

The same statement does not hold for ternary words, however, we conjecture that only slightly more complicated, \emph{dihedral} permutations suffice for an analogous goal (Conjecture \ref{Conjecture Ternary Dihedral}). In general, we suspect that for every fixed alphabet, a linear number of permutations should be sufficient to express the analogous generalized shuffle square property (Conjecture \ref{Conjecture Main}). We support these conjectures by presenting results of various computer experiments as well as some theoretical observations. In the final part of the paper we propose some directions for future research and some connections to widely studied topics, like \emph{Gauss codes}, \emph{DNA sequencing}, and \emph{circle graphs}.

\section{Problems and results}

\subsection{Notation and terminology}
Let $\mathbb{A}$ be a fixed alphabet and let $W$ be a word over $\mathbb{A}$. A subword of $W$ is any word obtained by deleting some (possibly zero) letters of $W$. For example, $\mathtt{attic}$ is a subword of $\mathtt{m\color{blue}{at}\color{black}{hema}\color{blue}{tic}\color{black}{s}}$. A special type of subword consisting of consecutive letters is called a \emph{factor} of $W$. Thus, a word $F$ is a factor of $W$ if $W=PFS$ for some (possible empty) words $P$ and $S$ (called a~{\it prefix} and a \emph{suffix} of $W$, respectively). For instance, $\mathtt{thema}$ is a factor of $\mathtt{ma\color{blue}{thema}\color{black}{tics}}$. A \emph{square} is a word $W=UU$ for some nonempty word $U$ and a \emph{shuffle square} is a word that can be split into two identical disjoint subwords. For example, the word $\mathtt{\color{red}{hots}\color{blue}{hots}}$ is a square, while $\mathtt{\color{red}{tu}\color{blue}{t}\color{red}{e}\color{blue}{u}\color{red}{r}\color{blue}{er}}$ is a shuffle square, but not a square. Clearly, every letter in a shuffle square must occur an even number of times. We will call any word with this property an \emph{even} word.

In the present paper we introduce the following generalization of squares and shuffle squares. Let us start with looking at the word $\mathtt{rattar}$. It is not a square, but it is a~concatenation of two words that are in some sense similar --- one is just a \emph{reverse} of the other. To be more formal, let $\gamma$ be a permutation of the set $[n]=\{1,2,\ldots,n\}$ denoted as a sequence $\gamma=\gamma_1\gamma_2\cdots \gamma_n$, with $\gamma_i\in [n]$. If $W=w_1w_2\cdots w_n$ is a word of length $n$, then $\gamma(W)=w_{\gamma_1}w_{\gamma_2}\cdots w_{\gamma_n}$ is an effect of applying $\gamma$ to the word $W$. For instance, if $W=\mathtt{sword}$ and $\gamma=23451$, then $\gamma(W)=\mathtt{words}$.

Two words, $U=u_1u_2\cdots u_n$ and $V=v_1v_2\cdots v_n$, are called \emph{$\gamma$-similar} if $U=\gamma(V)$ or $V=\gamma(U)$. For example, the words $U=\mathtt{braze}$ and $V=\mathtt{zebra}$ are $\gamma$-similar with $\gamma=34512$. A \emph{$\gamma$-square} is just a word of the form $UV$, where $U$ and $V$ are $\gamma$-similar. A word $W$ is a~\emph{shuffle $\gamma$-square} if it can be split into two subwords that are $\gamma$-similar.

Of course, every even word is a shuffle $\gamma$-square for some permutation $\gamma$. One naturally wonders how ``complex'' this permutation must be. Our ultimate goal is to express all even $k$-ary words as shuffle $\gamma$-squares with some restricted class of permutations $\gamma$.

\subsection{Cyclic shuffle squares}
In this subsection we prove that every even binary word is a~shuffle $\gamma$-square, where $\gamma$ is a~cyclic permutation. The proof is based on a beautiful theorem of Goldberg and West \cite{Goldberg-West, Alon-West} on the fair splitting of necklaces, which is formulated below.

A \emph{necklace} is just a word $N$ over a finite alphabet. A \emph{fair splitting} of $N$ is a factorization $$N=N_1N_2\cdots N_k$$ and a partition of the resulting factors $N_i$ into two disjoint families, $\mathcal{F}_1$ and $\mathcal{F}_2$, such that the total number of occurrences of each letter in the first family is the same as in the second family. Clearly, a necessary and sufficient condition for a fair splitting is that $N$ is an even word (each letter occurs an even number of times). The theorem below states that the number of cuts needed to obtain a fair splitting is at most the number of different letters occurring in the necklace.

\begin{theorem}[Goldberg and West, \cite{Goldberg-West}]\label{Theorem Necklace Splitting}
Every even word with $k$ distinct letters has a fair splitting into at most $k+1$ factors.
\end{theorem}

An elegant proof of this result, using the celebrated Borsuk-Ulam theorem, was found by Alon and West \cite{Alon-West}. A more general version concerning fair splittings of necklaces into arbitrary number of parts was obtained by Alon \cite{Alon}.

We will make use of the following standard notation. For any letter $\mathtt{b}$ of the alphabet, let $\mathtt{b}^n$ denote the word made of a run of $n$ copies of $\mathtt{b}$, i.e., $$\mathtt{b}^n=\underbrace{\mathtt{bb\cdots b}}_n.$$ In general, for a word $W$, we write $W^n$ for a concatenation of $n$ copies of $W$, i.e., $$W^n=\underbrace{WW\cdots W}_n.$$ Recall that a~permutation $\gamma$ is \emph{cyclic} if it has the form $$\gamma=i(i+1)\cdots n12\cdots (i-1)$$ for some $i\in [n].$
\begin{theorem}\label{Theorem Cyclic Shuffle Squares}
	For every even binary word $W$ there exists a cyclic permutation $\gamma$ such that $W$ is a shuffle $\gamma$-square.
\end{theorem}
\begin{proof}
Let $W$ be any even binary word. Assume that the number of $\mathtt{0}$'s in $W$ is $2r$ and the number of $\mathtt{1}$'s in $W$ is $2s$. Then, by Theorem \ref{Theorem Necklace Splitting}, $W$ can be written as $W=XVY$, where the word $XY$ has $r$ $\mathtt{0}$'s and $s$ $\mathtt{1}$'s, and the same is true for the word $V$. Let us denote by $X_{\mathtt{0}}$ and $X_{\mathtt{1}}$ the subwords of $X$ consisting of all $\mathtt{0}$'s and all $\mathtt{1}$'s, respectively. Let $Y_{\mathtt{0}}$, $Y_{\mathtt{1}}$ and $V_{\mathtt{0}}$, $V_{\mathtt{1}}$ denote similarly defined subwords of $Y$ and $V$, respectively.

Consider now two subwords of $W$, namely, $A=X_{\mathtt{0}}V_{\mathtt{1}}Y_{\mathtt{0}}$ and $B=X_{\mathtt{1}}V_{\mathtt{0}}Y_{\mathtt{1}}$. It is clear that $$\gamma_1(A)=Y_{\mathtt{0}}X_{\mathtt{0}}V_{\mathtt{1}}=\mathtt{0}^r\mathtt{1}^s\ \ \ \mbox{ and }\ \ \ \gamma_2(B)=V_{\mathtt{0}}Y_{\mathtt{1}}X_{\mathtt{1}}=\mathtt{0}^r\mathtt{1}^s,$$ for some cyclic permutations $\gamma_1$ and $\gamma_2$. It follows that $\gamma_1(A)=\gamma_2(B)$, and consequently $A=\gamma(B)$, for some cyclic permutation~$\gamma$. This completes the proof.
\end{proof}

The above quite surprising theorem leads us to consider the stronger statement: namely, is it true that every even binary word is cyclically similar to a pure shuffle square? For instance, let us take the word $W=\mathtt{10010110}$. It is not a shuffle square itself, but moving its first two letters to the end gives the word $U=\mathtt{\color{red}01\color{blue}01\color{red}10\color{blue}10}$, which is a shuffle square. Is it possible that every even binary word $W$ has a factorization $W=XY$ such that $U=YX$ is a shuffle square? As it turns out, such a statement is false, and the smallest counterexample is the word
$$\mathtt{000001001111000011101111}.$$

The above word is the only counterexample of length $24$ up to cyclic shifts, reverses, and permutations of the alphabet. We call such word a {\it shuffle anti-sqaure}. More shuffle anti-squares are presented in Appendix \ref{apA}.

However, we can distinguish a certain special type of words $W$ for which the considered statement is true.

\begin{proposition}\label{Proposition Four 1's Cyclic}
	For every even binary word $W$ with at most four $\mathtt{1}$'s, there exists a cyclic permutation $\gamma$ such that the word $\gamma(W)$ is a shuffle square.
\end{proposition}
\begin{proof}
In our argumentation we will often use the following simple observations. Suppose that $W=W_1W_2\cdots W_n$ and each of the words $W_i$ is a shuffle square. Then $W$ is a shuffle square, too. Also, if $W$ is a shuffle square, then so is its reverse.

Let $W$ be a binary word with two $\mathtt{1}$'s and even number of $\mathtt{0}$'s. We may assume that it has the form of $W=\mathtt{1}\mathtt{0}^a\mathtt{1}\mathtt{0}^b$ for some integers $a$ and $b$. If $a\leqslant b$, then $W=(\mathtt{10}^a\mathtt{10}^a)(\mathtt{0}^{b-a})$, where $b-a$ is even. So, $W$ is a concatenation of two squares, hence a shuffle square itself. If $a>b$, then the cyclic shift $W'$ of $W$ is a shuffle square, since it can be written as $W'=(\mathtt{0}^{a-b})(\mathtt{0}^b\mathtt{10}^b\mathtt{1})$.

Now let $W$ be an even binary word with four $\mathtt{1}$'s. Clearly, we may assume that $W$ starts with a block of $\mathtt{1}$'s. Then the proof splits into four cases, depending on the length of said block.
\begin{case}[$W=\mathtt{11110}^{2r}$]
\emph{Clearly $W$ is a concatenation of two squares, so a shuffle square itself.}
\end{case}
\begin{case}[$W=\mathtt{1110}^a\mathtt{10}^b$]
	\emph{If $a\leqslant b$, then we can write $$W=(\mathtt{11})(\mathtt{10}^a\mathtt{10}^a)(\mathtt{0}^{b-a}),$$ where $b-a$ is even. Thus $W$ is a concatenation of three squares. If $a>b$, then the cyclic shift of $W$ factorizes into three squares, namely $$W'=(\mathtt{11})(0^{a-b})(0^b10^{b}1).$$ So $W'$ is a shuffle square.}
\end{case}
\begin{case}[$W=\mathtt{110}^a\mathtt{10}^b\mathtt{10}^c$]
	\emph{Consider a cyclic shift $W'$ of $W$, i.e. $$W'=\mathtt{10}^a\mathtt{10}^b\mathtt{10}^c\mathtt{1}.$$ Denote $r=\frac{a+b+c}{2}$. If $a\leqslant r \leqslant a+b$, then $W'$ can be split into $$W'=\color{red}{\mathtt{10}}^a\color{blue}{\mathtt{1}}\color{red}{0^{r-a}}\color{blue}{\mathtt{0}^{a+b-r}}\color{red}{\mathtt{1}}\color{blue}{\mathtt{0}^c\mathtt{1}}.$$ Clearly, both subwords are identical and equal to $\mathtt{10}^r\mathtt{1}$. Suppose now that $r<a$ and consider the original word $W$. We may assume that $b>0$ since otherwise we are back to the case with two $\mathtt{1}$'s. We claim that we may always factorize $W$ into at most three shuffle squares, as $W=S_1S_2S_3$. Indeed, if $c$ is even, then we may take $$S_1=\mathtt{11}\mathtt{0}^{a-b},\,S_2=\mathtt{0}^b\mathtt{1}\mathtt{0}^b\mathtt{1}\ \mbox{ and }\ S_3=\mathtt{0}^c.$$ If $c$ is odd, then we may take $$S_1=\mathtt{11}\mathtt{0}^{a-b+1},\,S_2=(\mathtt{0}^{b-1}\mathtt{1}\mathtt{0})(\mathtt{0}^{b-1}\mathtt{1}\mathtt{0})\ \mbox{ and }\ S_3=\mathtt{0}^{c-1}.$$}
\end{case}
\begin{case}[$W=\mathtt{10}^a\mathtt{10}^b\mathtt{10}^c\mathtt{10}^d$]
	\emph{We may assume that all numbers, $a,b,c,d$, are positive, since otherwise we are back to one of the previous cases. Also, let $a$ be the least among all of the four. Then we may treat the word $\mathtt{I}=\mathtt{10}^a$ as a new letter, and consider the word $$U=\mathtt{I}\mathtt{I0}^{b-a}\mathtt{I0}^{c-a}\mathtt{I0}^{d-a}$$ as a word over a binary alphabet $\{\mathtt{I},\mathtt{0}\}$. Since $U$ has even number of $\mathtt{0}$'s, we are done by the previously considered cases.}
\end{case}
The proof of the proposition is thus completed.
\end{proof}

The above results naturally lead us to the investigation of the possible number of cyclic permutations that result in a shuffle square for a given binary word. Let $W$ be a word of length $l$. Let $\overline{W}$ be a multiset containing $l$ elements, which are all circular shifts of $W$ (i.e., words $\gamma(W)$ for cyclic permutations $\gamma$). For example
$$\overline{\mathtt{0011}}=\{\mathtt{0011},\mathtt{0110},\mathtt{1100},\mathtt{1001}\} \ \ \ \ \ \mbox{ and }\ \ \ \ \ \overline{\mathtt{0101}}=\{\mathtt{0101},\mathtt{1010},\mathtt{0101},\mathtt{1010}\}.$$

For every binary even word $W$ let $s( W)$ denote a number of elements of $\overline{W}$, which are shuffle squares. We have
$$s(0011)=2\ \mbox{ and }\ s(0101)=4.$$
Of course $s(W)=l$ if $W$ is a square.

Let $T_{2n}$ denote a set of all binary even words of length $2n$. Let $s_{2n}$ denote the minimal value of $S$ for binary even words of length $2n$:
$$s_{2n}=\min_{ W\in T_{2n}}s( W).$$

Let $\overline{S_{2n}}$ denote the set of even binary words $W$ of length $2n$ such that $s(W)=s_{2n}$. Let us notice that if $W\in\overline{S_{2n}}$, than $W'\in\overline{S_{2n}}$ if $V=\gamma{W}$ and $\gamma$ is a cyclic permutation of~$W$, a~reverse of $W$ or a permutation of the alphabet. Thus let $S_{2n}$ be a set of even binary words $W$ of length $2n$ such that $s(W)=s_{2n}$ up to the cyclic permutations, reverses and permutation of the alphabet. The values and some properties of $s_{2n}$ for small $n$ are given in Table \ref{T1}. Values of $s_{2n}$ for $n\geqslant12$ are pointing out the existence of suffle anti-squares of those lengths --- listed in Appendix \ref{apA}. 

	\begin{table}\centering
	\begin{tabular}{|c||c|c|c|c|c|c|c|c|c|c|c|c|c|c|c|}\hline
$2n$&$2$&$4$&$6$&$8$&$10$&$12$&$14$&$16$&$18$&$20$&$22$&$24$&$26$&$28$\\\hline
$s_{2n}$&$2$&$2$&$3$&$3$&$3$&$3$&$3$&$2$&$2$&$2$&$1$&$0$&$0$&$0$\\\hline
$|S_{2n}|$&$1$&$1$&$1$&$1$&$1$&$1$&$2$&$1$&$3$&$13$&$12$&$1$&$26$&$103$\\\hline
	    \end{tabular}
    \vskip0.5cm
	   	\caption{The values of $s_{2n}$ and sizes of the sets $S_{2n}$ for $n\leqslant14$.}\label{T1}
	    \end{table}
Thus we state the following supposition.
\begin{conjecture}\label{newGPP}
There exists a shuffle anti-square of length $2n$ for every $n\geqslant12$.
\end{conjecture}
Let us notice that the above conjecture is equivalent to the statement that $s_{2n}=0$ holds for every $n\geqslant12$.

\subsection{Dihedral shuffle squares}
It is natural to wonder to what extent the analogue of Theorem \ref{Theorem Cyclic Shuffle Squares} is true for larger alphabets. For ternary words the statement is no longer true, as shows the example of the word $W=\mathtt{012210}$. Indeed, there are only two pairs of similar words into which $W$ can be split, namely $(\mathtt{012},\mathtt{210})$ and $(\mathtt{021,\mathtt{120}})$, which are however not $\gamma$-similar for any cyclic $\gamma$. Nevertheless, it seems plausible that not much bigger class of permutations should be sufficient for ternary words.

Let $S_n$ denote the symmetry group consisting of all permutations of $[n]$. Recall that a \emph{dihedral group $D_n$} is a subgroup of $S_n$ consisting of all permutations corresponding to symmetries of a regular polygon with $n$ vertices. In particular, $|D_n|=2n$. Notice also that $D_3=S_3$. 

A \emph{dihedral shuffle square} of length $2n$ is any shuffle $\gamma$-square, where $\gamma\in D_n$. Based on intuition and some numerical experiments we propose the following conjecture.

\begin{conjecture}\label{Conjecture Ternary Dihedral}
	Every even ternary word is a shuffle $\gamma$-square for a dihedral permutation $\gamma$.
\end{conjecture}

A natural approach would be to use Theorem \ref{Theorem Necklace Splitting}. In the ternary case it says that every word can be written as $W=XYZT$ so that $XZ$ is similar to $YT$ (with some of the factors possibly empty). The crucial structures here, as in the binary case, are so called \emph{abelian squares}, i.e., words of the form $W=UV$, where $U$ is $\gamma$-similar to $V$ for some permutation~$\gamma$. To adress Conjecture \ref{Conjecture Ternary Dihedral} one should perhaps start with proving it first for abelian squares, which already looks like a nontrivial task.

Let us also mention that an analogue of Conjecture \ref{Conjecture Main} for ternary words with dihedral instead of cyclic permutations is no longer true as well. More precisely, it is not true that every even ternary word can be turned into a shuffle square by applying a dihedral permutation to the whole word. The counter example is the same word $W=\mathtt{001221}$.

\subsection{Large alphabets and canonical words}Perhaps the most intriguing problem is to find out what happens with generalized shuffle squares for larger alphabets. Let $\mathbb{A}$ be a fixed alphabet of size $k$. Let $E_n$ denote the set of all even $k$-ary words of length $n$ and let $S_n$ be the set of all permutations of $[n]$. Consider a bipartite graph $G_n$ with bipartition classes $E_n$ and $S_n$ with the edge set consisting of all pairs $(W,\gamma)$ such that $W$ is a shuffle $\gamma$-square. For a subset of vertices $X$ of $G_n$, we denote by $N(X)$ the set of all vertices that have a neighbor in $X$. We are interested in subsets $X\subseteq S_n$ such that $N(X)=E_n$. We call them \emph{covering} sets of permutations.

The main problem is to determine the minimum size of a covering set for given $k$ and $n$. Let us denote this quantity by $m_k(n)$. Theorem \ref{Theorem Cyclic Shuffle Squares} states that for every $n$, the set of all cyclic permutations of $[n]$ is a covering set. Hence, we have $m_2(n)\leqslant n$ for every natural $n$. Actually, this bound could be lowered by the obvious observation that in any minimal covering set there are no pairs of mutually inverse (distinct) permutations (since two words are $\gamma$-similar if and only if they are $\gamma^{-1}$-similar).  Thus, we have in fact $m_2(n)\leqslant \frac{1}{2}n+1$ for all $n\geqslant 1$. This leads to our main conjecture.

\begin{conjecture}\label{Conjecture Main}
	For every $k\in \mathbb{N}$ there exists a constant $c_k$ such that $m_k(n)\leqslant c_kn$ for all $n\geqslant 1$.
\end{conjecture}

Notice that Conjecture \ref{Conjecture Ternary Dihedral} implies the above statement for $k=3$ with constant $c_3=2$. By the same argument with inverse permutations this bound could be lowered to $m_3(n)\leqslant \frac{3}{2}n+1$ provided that Conjecture \ref{Conjecture Ternary Dihedral} is true.

The problem of precise determination of the number $m_k(n)$ seems challenging even in the simplest case of $n=k$. By computer experiments we have determined exact values of $m_k(k)$ for a few initial $k$.  

\begin{proposition}
	We have $m_2(2)=2$, $m_3(3)=5$, $m_4(4)=14$.
\end{proposition}

In our computations we have made some simple reductions that we explain below. In order to transparently distinguish words from permutations we will use ordinary capital letters $\mathtt{A},\mathtt{B},\mathtt{C},\ldots$ for symbols of the alphabet.

Firstly, let us notice that in calculating $m_k(k)$ one may restrict to $k$-ary words in which every letter occurs exactly twice. Moreover, one may identify words differing only by a permutation of the alphabet (since for two such words $U,W$ we have $N(U)=N(W)$). So, picking the lexicographically least word as a representative of the class of identified words we get a~collection $C_k$ of \emph{canonical words}. For instance, for $k=2$ we have only three canonical words: $\mathtt{AABB}$, $\mathtt{ABAB}$, and $\mathtt{ABBA}$, while for $k=3$ there are $15$ of them (see Table \ref{Table G'_3}). It is not hard to check that in general we have $$|C_k|=\frac{(2k)!}{k!\cdot 2^k}.$$

Reducing the set of permutations from $S_k$ to $S'_k$ by picking one for each pair of mutually inverse permutations, gives the reduced bipartite graph $G'_k$ on the bipartition classes $C_k$ and $S'_k$. Clearly, every covering set of permutations in $G'_k$ is at the same time a covering set in $G_k$.

Let us see how to find all neighbors of a canonical word in the graph $G'_k$ through the~example of the word $W=\mathtt{ABABCC}$. In order to do that we need to consider all possible decompositions of $W$ into a pair of similar words $W'$ and $W''$ (blue and red), and for each such pair determine the permutation transforming $W'$ into $W''$ (see Table \ref{Table ABABCC}). Actually, as seen in Table \ref{Table ABABCC}, there are only two different pairs of words (up to the order) one may obtain from $W$, namely, $(\mathtt{ABC},\mathtt{ABC})$ and $(\mathtt{ABC},\mathtt{BAC})$, which give two corresponding permutations, $123$ and $213$.

\begin{table}

\begin{center}
	\begin{tabular}{|c|c|c||c|c|c|}\hline
		$\mathtt{\textcolor{blue}{AB}\textcolor{red}{AB}\textcolor{blue}{C}\textcolor{red}{C}}$&$\mathtt{\textcolor{blue}{ABC}}\to \mathtt{\textcolor{red}{ABC}}$&123&$\mathtt{\textcolor{red}{A}\textcolor{blue}{BA}\textcolor{red}{B}\textcolor{blue}{C}\textcolor{red}{C}}$&$\mathtt{\textcolor{blue}{BAC}}\to \mathtt{\textcolor{red}{ABC}}$&213\\\hline
		$\mathtt{\textcolor{blue}{AB}\textcolor{red}{ABC}\textcolor{blue}{C}}$&$\mathtt{\textcolor{blue}{ABC}}\to \mathtt{\textcolor{red}{ABC}}$&123&$\mathtt{\textcolor{red}{A}\textcolor{blue}{BA}\textcolor{red}{BC}\textcolor{blue}{C}}$&$\mathtt{\textcolor{blue}{BAC}}\to \mathtt{\textcolor{red}{ABC}}$&213\\\hline
		$\mathtt{\textcolor{blue}{A}\textcolor{red}{BA}\textcolor{blue}{BC}\textcolor{red}{C}}$&$\mathtt{\textcolor{blue}{ABC}}\to \mathtt{\textcolor{red}{BAC}}$&213&$\mathtt{\textcolor{red}{AB}\textcolor{blue}{ABC}\textcolor{red}{C}}$&$\mathtt{\textcolor{blue}{ABC}}\to \mathtt{\textcolor{red}{ABC}}$&123\\\hline
		$\mathtt{\textcolor{blue}{A}\textcolor{red}{BA}\textcolor{blue}{B}\textcolor{red}{C}\textcolor{blue}{C}}$&$\mathtt{\textcolor{blue}{ABC}}\to \mathtt{\textcolor{red}{BAC}}$&213&$\mathtt{\textcolor{red}{AB}\textcolor{blue}{AB}\textcolor{red}{C}\textcolor{blue}{C}}$&$\mathtt{\textcolor{blue}{ABC}}\to \mathtt{\textcolor{red}{ABC}}$&123\\\hline
		
	\end{tabular}\vskip0.5cm
\caption{The canonical word $\mathtt{ABABCC}$ and all its decompositions leading to two neighboring permutations $123$ and $213$ in the graph $G'_3$.}\label{Table ABABCC}
\end{center}
\end{table}

\begin{table}
	\begin{tabular}{|l|l||l|l||l|l|}\hline
		word&permutations&word&permutations&word&permutations\\\hline
		$\mathtt{AABCBC}$&$123,132$&$\mathtt{ABABCC}$&$123,213$&$\mathtt{ABCCAB}$&$231,321$\\\hline
		$\mathtt{AABCCB}$&$132$&$\mathtt{ABCABC}$&$123,231$&$\mathtt{ABCCBA}$&$321$\\\hline
		$\mathtt{AABBCC}$&$123$&$\mathtt{ABCACB}$&$132,231,321$&$\mathtt{ABBACC}$&$213$\\\hline
		$\mathtt{ABACBC}$&$123,132,213,231$&$\mathtt{ABCBAC}$&$213,231,321$&$\mathtt{ABBCAC}$&$213,231$\\\hline
		$\mathtt{ABACCB}$&$132,231$&$\mathtt{ABCBCA}$&$231,321$&$\mathtt{ABBCCA}$&$231$\\\hline
	\end{tabular}\vskip0.5cm
	\caption{Permutation neighborhoods of canonical words in the graph $G'_3$.}\label{Table G'_3}
\end{table}

In Table \ref{Table G'_3} we see all other neighborhoods in the graph $G'_3$ obtained in this way. A picture of the actual graph is presented in Figure \ref{Graph G'_3}. Notice that in $S_3$ there are four convolutions and one pair of mutually inverse cycles, namely $231$ and $312$. We picked the former permutation to the set $S'_3$. In Table \ref{Table G'_3} we see that there are five words with single neighbors exhausting all elements of $S'_3$. So the only minimum covering set is the whole $S'_3$ and therefore $m_3(3)=5$.

\begin{figure}[ht]
	
	\begin{center}
		
		\resizebox{15cm}{!}{
			
			\includegraphics{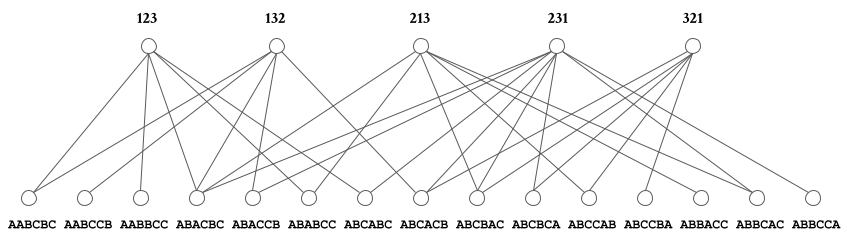}
			
		}
		\caption{Graph $G'_3$ drawn from Table \ref{Table G'_3}.}
		\label{Graph G'_3}
	\end{center}
\end{figure}

In case of $k=4$ the situation is slightly more nuanced. By computer search we have found a minimal covering set consisting of $14$ permutations (see Table \ref{Table 14 permutations}), which proves that $m_4(4)=14$. It has intriguing structure, namely, there are six inversions plus identity (in the first column of Table \ref{Table 14 permutations}), and two other convolutions, plus four $3$-cycles, plus one $4$-cycle (in the second column of Table \ref{Table 14 permutations}). Notice that all permutations (up to inverses) having a fixed point are present in this set.

Let us also remark that canonical words that are shuffle squares have simple characterization.

\begin{proposition}
	A canonical word $W\in C_k$ is a shuffle square if and only if it has no subword of the form $\mathtt{XYYX}$, where $\mathtt{X}$ and $\mathtt{Y}$ are single letters. Moreover, the number of shuffle squares in $C_k$ is equal to the Catalan number $\frac{1}{k+1}\binom{2k}{k}$.
\end{proposition}
\begin{proof}
	
	Let $C_k$ be the set of canonical words over alphabet $\mathbb{A}=\{\mathtt{A}_1,\mathtt{A}_2,\ldots,\mathtt{A}_k\}$. First notice that any shuffle square in $C_k$ must be splittable into blue and red copies of the word $U_k=\mathtt{A}_1\mathtt{A}_2\cdots \mathtt{A}_k$, by the lexicographic condition in the definition of canonical words. Notice also that we may always find blue and red copies of $U_k$ in $W$ so that the number of blue letters is not lesser than the number of red letters in every prefix of $W$. This can be done by coloring blue the first occurrence of each letter $\mathtt{A_i}$. This establishes a one-to-one correspondence between the set of shuffle squares in $C_k$ and the set of \emph{Dyck words} on the alphabet of two colors $\{\mathtt{b},\mathtt{r}\}$. For instance, the canonical word $\mathtt{AABCBC}$ can be split to $\color{blue}\mathtt{A}\color{red}\mathtt{ABC}\color{blue}\mathtt{BC}$, which gives the color pattern $\mathtt{brrrbb}$ (which is \emph{not} a Dyck word), but another splitting $\color{blue}\mathtt{A}\color{red}\mathtt{A}\color{blue}\mathtt{BC}\color{red}\mathtt{BC}$ gives the correct color pattern $\mathtt{brbbrr}$. Thus, the second assertion of the proposition is proved.
	
	For the first assertion, let $W$ be a shuffle square in $C_k$ and assume it has a subword $\mathtt{XYYX}$, with $\mathtt{X},\mathtt{Y}\in \mathbb{A}$. In a blue-red decomposition of $W$ there are only two possibilities of color placement on this subword (up to symmetry): $\color{blue}\mathtt{X}\mathtt{Y}\color{red}\mathtt{YX}$ or $\color{blue}\mathtt{X}\color{red}\mathtt{Y}\color{blue}\mathtt{Y}\color{red}\mathtt{X}$. In both cases we get a subword $\color{blue}\mathtt{XY}$ in blue and $\color{red}\mathtt{YX}$ in red. So, the whole blue and red subwords of $W$ cannot be the same.
	
	On the other hand, suppose that $W\in C_k$ is a canonical word not containing any subwords of the forbidden form. In particular, between two occurrences of the last letter $\mathtt{A}_k$, each other letter occurs once. 
	Then, we may delete the last pair of letters from $W$, apply induction to the shorter word $W'\in C_{k-1}$, and then append the deleted letters to the blue and red subwords of $W'$ accordingly.
\end{proof}
In graph theoretic terminology the above result states that the vertex corresponding to the identity permutation in the graph $G'_k$ has degree equal to the Catalan number $\frac{1}{k+1}\binom{2k}{k}$. As seen in Figure \ref{Graph G'_3} this is not the maximum vertex degree in $G'_3$. What is the maximum vertex degree in $G'_k$?
\begin{table}
		\begin{tabular}{|l||l|}\hline
			$1234=\text{id}$&$2143=(12)(34)$\\\hline
			$2134=(12)$&$4321=(14)(23)$\\\hline
			$3214=(13)$&$1342=(234)$\\\hline
			$4231=(14)$&$3241=(134)$\\\hline
			$1324=(23)$&$2431=(124)$\\\hline
			$1432=(24)$&$2314=(123)$\\\hline
			$1243=(34)$&$2341=(1234)$\\\hline
		\end{tabular}\vskip0.5cm
	\caption{Minimal covering set of $14$ permutations showing that $m_4(4)=14$.}\label{Table 14 permutations}
\end{table}

\subsection{The number of binary shuffle squares with two $\mathtt{1}$'s}\label{bss}
	
In this subsection we investigate shuffle squares with fixed number of $\mathtt{1}$'s. Let $|A|$ be the number of elements of the set $A$. Let $\mathtt{B}(2n,2k)$ be the set of all binary shuffle squares of the length $2n$ with the number of $\mathtt{1}$'s equal to $2k$. For example,
\begin{align*}
	\mathtt{B}(2,2)&=\{\mathtt{11}\},\\
	\mathtt{B}(4,2)&=\{\mathtt{0011},\,\mathtt{0101},\,\mathtt{1010},\,\mathtt{1100}\}.
\end{align*}
Let us notice that $|\mathtt{B}(2n,0)|=|\mathtt{B}(2n,2n)|=1$ and that
$$|\mathtt{B}(2n,2k)|=|\mathtt{B}(2n,2n-2k)|$$
for all $n,k$ such that $0\leqslant k\leqslant n$.
	
\begin{theorem}\label{ThChar}
	Let $W$ be a binary word of length $2n$ with exactly two $\mathtt{1}$'s. The~word $W$ is an element of $\mathtt{B}(2n,2)$ if and only if it contains one of the following factors:
	\begin{itemize}
		\item a prefix $\mathtt{0\cdots011}$ which has even length,
		\item a word $\mathtt{10\cdots01}$ of length $l$ such that $3\leqslant l\leqslant n+1$.
	\end{itemize}
\end{theorem}

\begin{proof}
	
	Firstly, we will show that the words with proposed factors are in fact shuffle squares. If $W$ has the first type of factor, namely $$W=\mathtt{0}^{2p}\mathtt{11}\mathtt{0}^{2s},$$ then $W$ can be split into two identical disjoint subwords $$W'=W''=\mathtt{0}^p\mathtt{1}\mathtt{0}^s.$$  
	For $W$ containing the factor $\mathtt{10\cdots01}$ we consider two cases, depending on the parity of the~number of $\mathtt{0}$'s between $\mathtt{1}$'s.
	\begin{dupa}[$\displaystyle W=\mathtt{0}^{2p}\mathtt{1}\mathtt{0}^{2k}\mathtt{1}\mathtt{0}^{2s}$]
		\emph{If $p\geqslant k$, then $W$ can be split into two indentical disjoint words $$W'=W''=\mathtt{0}^{p+k}\mathtt{1}\mathtt{0}^s.$$ If $p<k$, then
		$$W'=W''=\mathtt{0}^p\mathtt{1}\mathtt{0}^{k+s}.$$}
	\end{dupa}

\begin{dupa}[$W=\mathtt{0}^{2p+1}\mathtt{1}\mathtt{0}^{2k+1}\mathtt{1}\mathtt{0}^{2s-2}$ or $\displaystyle W=\mathtt{0}^{2p-2}\mathtt{1}\mathtt{0}^{2k+1}\mathtt{1}\mathtt{0}^{2s+1}$]
	\emph{Let us notice that it is sufficient to consider only $W$ with odd number of $\mathtt{0}$'s before the~first~$\mathtt{1}$, since the argument in the remaining case is analogous. For $p\geqslant k$ the subwords are $$W'=W''=\mathtt{0}^{p+k+1}\mathtt{1}\mathtt{0}^{s-1},$$
	and for $p<k$ we have
	$$W'=W''=\mathtt{0}^p\mathtt{1}\mathtt{0}^{k+s}.$$}
\end{dupa}
Thus, in both cases we are able to split the considered word into two identical disjoint subwords, so the word is a shuffle square.

Now let us show that $\mathtt{B}(2n,2)$ does not contain any other shuffle squares. If $W$ contains a~prefix $\mathtt{0\cdots011}$ of odd length, then any two disjoint subwords of $W$ of length $n$ contain prefixes with different numbers of $\mathtt{0}$'s before the first occurrence of $\mathtt{1}$, and so they are not identical. For the remaining case let us assume that $$W=\mathtt{0}^p\mathtt{1}\mathtt{0}^k\mathtt{1}\mathtt{0}^s,$$ where $k\geqslant n$ and that $W$ can be split into two identical subwords $$W'=W''=\mathtt{0}^{p'}\mathtt{1}\mathtt{0}^{s'}.$$ The length of $W$ is $2n$, so from $k\geqslant n$ we get $p+s\leqslant n-2$. Let us notice that $p'\leqslant p$ and $s'\leqslant s$. Thus $p'+s'\leqslant n-2$ and finally the length of $W'$ is not greater than $(n-1)$ which results in a contradiction with the fact that the length of $W'$ and $W''$ equals to $n$.
\end{proof}

\begin{corollary}\label{CorTrian}	
	The number of all binary shuffle squares of length $2n$ with exactly two $\mathtt{1}$'s equals $$\frac{3n(n-1)}2+1,$$ for every $n\geqslant 1$.	
\end{corollary}

\begin{proof}
	The statement obviously holds for $n=1$. Let us assume $n\geqslant2$. Theorem \ref{ThChar} states that there are two separate types of shuffle squares of length $2n$ which contain exactly two $\mathtt{1}$'s:
	\begin{itemize}
		\item $\mathtt{0}^{2k}\mathtt{11}\mathtt{0}^{2n-2k-2}$, for $0\leqslant k\leqslant n$.
		
		It is easy to show that there are exactly $n$ such shuffle squares.
		
		\item $\mathtt{0}^p\mathtt{1}\mathtt{0}^{k}\mathtt{10}^{2n-p-k-2}$, for $1\leqslant k\leqslant n-1$ and $p\leqslant2n-k-2$.
		
		Let us notice that for the fixed $k$ the number of shuffle squares of this type is equal to $(2n-k-1)$. Thus the number of all such shuffle squares for a given $n$ is equal to $$\sum\limits_{k=1}^{n-1}(2n-k-1).$$
	\end{itemize}
	Finally, we get
	$$|\mathtt{B}(2n,2)|=n+\sum\limits_{k=1}^{n-1}(2n-k-1)=\frac{3n(n-1)}2+1.$$
\end{proof}

Let us notice that the obtained numbers of those particular shuffle squares form the sequence A005448 of centered triangular numbers in the OEIS \cite{OEIS}. Initial terms of this sequence can be found in the third row of Table \ref{TabNum}. The last row of presented table contains initial terms of the sequence A191755 \cite{OEIS2} of the number of different binary shuffle squares of the length $2n$. It would be of great value to find at least the asymptotic order of other sequences $|\mathtt{B}(2n,2k)|$ for fixed $k\in \mathbb{N}$.

\begin{table}
\begin{tabular}{|c|c|c|c|c|c|c|c|c|c|c|}\hline
\backslashbox{$2k$}{$2n$}&2&4&6&8&10&12&14&16&18&20\\\hline
0&1&1&1&1&1&1&1&1&1&1\\\hline
2&1&4&10&19&31&46&64&85&109&136\\\hline
4&0&1&10&42&128&306&633&1169&1997&3199\\\hline
6&0&0&1&19&128&562&1853&5041&11914&25331\\\hline
8&0&0&0&1&31&306&1853&8040&27965&82208\\\hline
10&0&0&0&0&1&46&633&5041&27965&120718\\\hline
12&0&0&0&0&0&1&64&1169&11914&82208\\\hline
14&0&0&0&0&0&0&1&85&1997&25331\\\hline
16&0&0&0&0&0&0&0&1&109&3199\\\hline
18&0&0&0&0&0&0&0&0&1&136\\\hline
20&0&0&0&0&0&0&0&0&0&1\\\hline
$\sum$&2&6&22&82&320&1268&5102&20632&83972&342468\\\hline
\end{tabular}\vskip0.5cm
\caption{$|\mathtt{B}(2n,2k)|$ for small values of $k$ and $n$. The last row contains the~number of all shuffle squares for given values of $2n$.}\label{TabNum}
\end{table}
 
\section{Final remarks}\label{dis}
\begin{figure}[ht]
	
	\begin{center}
		
		\resizebox{7cm}{!}{
			
			\includegraphics{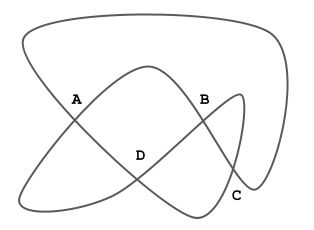}
			
		}
		\caption{A curve $C$ with a Gauss code $W=\mathtt{ABCADCBD}$.}
		\label{Gauss}
	\end{center}
\end{figure}

Concluding the paper, let us mention some possible applications of generalized shuffle squares in seemingly distant fields. Below we present three different topics in which canonical words play a crucial role.

\subsection{Gauss codes} Gauss \cite{Gauss} studied the following problem. Suppose that we are given a closed curve $C$ in the plane that crosses itself in a finite number of points (each point has multiplicity two). Then, denoting the crossing points by letters of some alphabet and traversing the curve in some direction gives a word in which every letter occurs exactly twice (see Figure \ref{Gauss}).

Nowadays, these words are called \emph{Gauss codes}. Notice that it is sufficient to restrict to canonical words (in our sense). Gauss himself found a necessary condition for a word to be such code, however the full characterization was presented much later (see \cite{ShtyllaTZ} for a~survey). It is natural to wonder what the minimal set of permutations is, covering the set of Gauss codes. It could also be interesting to investigate the related bipartite graph, that is, a~subgraph of $G'_k$ restricted to the set of Gauss codes.
\begin{figure}[ht]
	
	\begin{center}
		
		\resizebox{6cm}{!}{
			
			\includegraphics{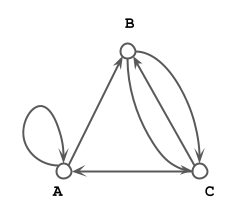}
			
		}
		\caption{Digraph $D_W$ for the word $W=\mathtt{AABCBC}$.}
		\label{Euler}
	\end{center}
\end{figure}
\subsection{DNA sequencing} One of the methods of reconstructing a long DNA string from its short substrings is \emph{sequencing by hybridization} (see \cite{ArratiaBCS}). For a given string, the number of possible reconstructions is determined by the pattern of repeated substrings. If each of $n$ substrings is assumed to occur at most twice in the original string, then the \emph{pattern of repeats} is a word $W$ of length $2n$ in which each of $n$ symbols occurs twice. Now the word $W$ induces a directed graph $D_W$, whose vertices are the letters corresponding to the substrings, while the arcs are defined by pairs of successive letters of $W$. For instance, if $W=\mathtt{AABCBC}$, then the set of arcs consists of $\mathtt{AA}$, $\mathtt{AB}$, $\mathtt{BC}$, $\mathtt{CB}$, $\mathtt{BC}$, and $\mathtt{CA}$ (see Figure \ref{Euler}). The number of possible reconstructions of the original DNA string is thus equal to the number of Euler circuits of the graph $D_W$, called the \emph{Euler number} of the word $W$ (see \cite{ArratiaBCS}).

It seems plausible that splitting properties of canonical words are somehow encoded in these digraphs. It is not hard to verify, for instance, that canonical words with Euler number one coincide with shuffle squares. What are the canonical words with Euler number two? Can we detect them in our graph $G'_k$?

\subsection{Circle graphs} Another intriguing structures related to canonical words are \emph{circle graphs}. These are intersection graphs of chords of a circle with pairwise disjoint ends. This class of graphs has surprisingly rich structure and many intriguing connections and applications (e.g. DNA sequencing \cite{ArratiaBCS} or logic \cite{Courcelle}).
\begin{figure}[ht]
	
	\begin{center}
		
		\resizebox{13cm}{!}{
			
			\includegraphics{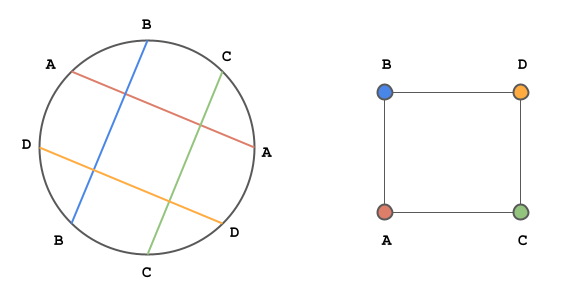}
			
		}
		\caption{The chord diagram of the word $\mathtt{ABCADCBD}$ and its circle graph.}
		\label{Circle Graph}
	\end{center}
\end{figure}
Denoting the ends of a chord by same letters leads to the canonical word (up to circular equivalence).

Again, one may hope that the structure of chord diagrams and circle graphs of canonical words reflects somehow their splitting properties. Perhaps the methods applied for studying these structures turn out to be fruitful there, too. For instance, our main Conjecture \ref{Conjecture Main} states that for every even binary word of length $2n$ written on a circle, there is a chord diagram, with each chord joining two $\mathtt{0}$'s or two $\mathtt{1}$'s, such that the corresponding canonical word is a~shuffle square.

	\appendix
	
	\section{List of shuffle ant-squares for small $n$}\label{apA}

The shuffle anti-square is an even word $W$ such that $W$ is not a shuffle square, and for every words $P$ and $S$ such that $W=PS$, the word $SP$ is not a shuffle square. Table $T1$ contains an information, that up to the cyclic shifts, reverses, and permutations of the alphabet, there is exactly 1 shuffle anti-square of length 24, 26 shuffle anti-squares of length 26 and 103 shuffle anti-squares of length 28. We present those words below.

The only shuffle anti-square of length 24 is
$$\mathtt{000001001111000011101111}.$$ 

Shuffle anti-squares of length 26 are
$$\begin{array}{cc}
\mathtt{00000001001110010000110111},&\mathtt{00000001001110010001100111},\\
\mathtt{00000001110100100010010111},&\mathtt{00000011010011011000011111},\\
\mathtt{00000011010110011000011111},&\mathtt{00000011010110101000011111},\\
\mathtt{00000011011010101000011111},&\mathtt{00000100011101000011011111},\\
\mathtt{00000100011110000011011111},&\mathtt{00000100011110000110101111},\\
\mathtt{00000100111100001101011011},&\mathtt{00000100111110001110011111},\\
\mathtt{00000101011010101000011111},&\mathtt{00000101011110100011101111},\\
\mathtt{00000101011110100011110111},&\mathtt{00000111100110010110011111},\\
\mathtt{00001111010010000111110111},&\mathtt{00001111100010000111110111},\\
\mathtt{00001111100010000111111011},&\mathtt{00001111101010001001110001},\\
\mathtt{00001111101100100111100101},&\mathtt{00001111101101000111110001},\\
\mathtt{00001111101111000111110101},&\mathtt{00001111101111000111111001},\\
\mathtt{00001111110111000111111001},&\mathtt{00001111111000110110111011}.
\end{array}$$

\newpage

Shuffle anti-squares of length 28 are

$$\begin{array}{cc}
\mathtt{0000000010001110010000110111},&\mathtt{0000000010011100100000110111},\\
\mathtt{0000000100001110000001101111},&\mathtt{0000000100011110000011101111},\\ 
\mathtt{0000000100100110100000110111},&\mathtt{0000000100101100100000110111},\\ 
\mathtt{0000000100110010100000101111},&\mathtt{0000000100110100100000101111},\\ 
\mathtt{0000000100111100000011101111},&\mathtt{0000000100111100000111001111},\\
\mathtt{0000000100111100010001101111},&\mathtt{0000000100111100100001101111},\\
\mathtt{0000000100111100100011001111},&\mathtt{0000000101111000010011011011},\\
\mathtt{0000000110010011011000011111},&\mathtt{0000000110010011101000011111},\\
\mathtt{0000000110010110011000011111},&\mathtt{0000000110010110101000011111},\\
\mathtt{0000000110011001011000011111},&\mathtt{0000000110011001100010011111},\\
\mathtt{0000000110011001100100011111},&\mathtt{0000000110011001101000011111},\\
\mathtt{0000000110011010011000011111},&\mathtt{0000000110011010101000011111},\\
\mathtt{0000000110011100011000011111},&\mathtt{0000000110011100101000011111},\\
\mathtt{0000000110100110110000011111},&\mathtt{0000000110101100000110011111},\\
\mathtt{0000000110101100110000011111},&\mathtt{0000000110101101010000011111},\\
\mathtt{0000000110110101010000011111},&\mathtt{0000000110111001100011000111},\\
\mathtt{0000000111001000011010011111},&\mathtt{0000000111001001001010011111},\\
\mathtt{0000000111001010001010011111},&\mathtt{0000000111001010001100011111},\\
\mathtt{0000000111001100001010011111},&\mathtt{0000000111001100001011001111},\\
\mathtt{0000000111001100001100011111},&\mathtt{0000000111010001001010011111},\\
\mathtt{0000000111010010001010011111},&\mathtt{0000000111010100001010011111},\\
\mathtt{0000000111100100100010011111},&\mathtt{0000000111101001000100101111},\\
\mathtt{0000001000011110000011011111},&\mathtt{0000001000111001000011011111},\\
\mathtt{0000001000111010000011011111},&\mathtt{0000001000111010000101101111},\\
\mathtt{0000001000111100000011011111},&\mathtt{0000001000111100000011101111},\\
\mathtt{0000001000111100000101101111},&\mathtt{0000001000111100000110101111},\\ 
\mathtt{0000001000111110000111011111},&\mathtt{0000001001011100000011101111},\\
\mathtt{0000001001011110100011101111},&\mathtt{0000001001101100000110101111},\\
\mathtt{0000001001111000001101011011},&\mathtt{0000001001111000001110011101},\\
\mathtt{0000001010101111000011110111},&\mathtt{0000001010110101010000011111},\\
\mathtt{0000001010111011000011101111},&\mathtt{0000001010111101000011101111},\\
\mathtt{0000001010111101000011110111},&\mathtt{0000001011011100111000011111},\\
\mathtt{0000001011101011000011011111},&\mathtt{0000001011101101000011011111},\\
\mathtt{0000001011110000110001110011},&\mathtt{0000001011110000110010110011},\\
\mathtt{0000001011110000110011010011},&\mathtt{0000001011111000011101100111},\\
\mathtt{0000001011111001000011000111},&\mathtt{0000001011111001000101000111},\\
\mathtt{0000001011111001000110000111},&\mathtt{0000001011111001001001000111},\\
\mathtt{0000001011111001001010000111},&\mathtt{0000001011111001001101001111},\\
\mathtt{0000001100011100000011011111},&\mathtt{0000001100101100110000011111},\\
\mathtt{0000001101111000100010010111},&\mathtt{0000001101111000100010100111},\\
\end{array}$$
$$\begin{array}{cc}
\mathtt{0000001101111010000011000111},&\mathtt{0000001101111100000100001111},\\
\mathtt{0000001101111100000101000111},&\mathtt{0000001101111100000110000111},\\
\mathtt{0000001101111100001000001111},&\mathtt{0000001101111100001001000111},\\
\mathtt{0000001101111100001010000111},&\mathtt{0000010000111010000011011111},\\
\mathtt{0000010000111010000101011111},&\mathtt{0000010000111100000101011111},\\
\mathtt{0000010000111100000110011111},&\mathtt{0000010000111110101000001111},\\
\mathtt{0000010001011100000110011111},&\mathtt{0000010001101100000101011111},\\
\mathtt{0000010001101100000110011111},&\mathtt{0000010001110100000101011111},\\
\mathtt{0000010010101100000110011111},&\mathtt{0000010011010110010000011111},\\
\mathtt{0000010011111000001100110011},&\mathtt{0000010011111000100001100111},\\
\mathtt{0000010011111001000001101011},&\mathtt{0000010101011101000011101111},\\
\mathtt{0000011001100110001000011111}.
\end{array}$$

	\bigskip
	\hrule
	\bigskip
	
	\noindent 2020 {\it Mathematics Subject Classification}: 68R15.
	
	\noindent \emph{Keywords: } combinatorics on words, shuffle squares, reverse shuffle squares.
	
	\bigskip
	\hrule
	\bigskip

\end{document}